\documentclass[12pt,a4paper]{article}

\usepackage{amssymb}
\usepackage{amsfonts}
\usepackage{amsmath}
\usepackage{amsthm}
\usepackage[english]{babel}
\usepackage[cp1250]{inputenc}
\usepackage[T1]{fontenc}
\usepackage[pdftex]{graphicx}

\newtheorem{thm}{Theorem}
\newtheorem{cor}{Corollary}
\newtheorem{lem}[thm]{Lemma}

\title{Bessel Function Model for Corneal Topography}
\author{Wojciech Okrasi\'nski, \L ukasz P\l ociniczak}

\begin{document}
\maketitle


\begin{abstract}
In this paper we consider a new nonlinear mathematical model for corneal topography formulated as two-point boudary value problem. We derive it from first physical principles and provide some mathematical analysis. The existence and uniqeness theorems are proved as well as various estimates on exact solution. At the end we fit the simplified model based on Modified Bessel Function of the First Kind with the real corneal data consisting of matrix of $123\times 123$ points and obtain an error of order of $1 \%$.
\end{abstract}

\section{Introduction}
Sight is the most important sense that we posses. It enables us to truly perceive the outer world. It is crucial to understand the mechanics of vision in order to treat various diseases that may occur and disturb regular seeing. The eye's main part responsible for about two-thirds of refractive power is the cornea. It is one of the most sensitive parts of human body and its various irregularities can cause many seeing disorders. Common diseases like myopia, hyperopia or astigmatism have their origin in some abnormalities in corneal geometry. Thus, precise knowledge of corneal shape is very important.

We can distinguish several mathematical models of cornea which are presently in use. Most commonly, cross-section of cornea is described by some conic section \cite{Conic} (for example, parabola or ellipse). The drawback of these models is that they have almost no physical motivation or derivation behind them. Probably the most appropriate model of corneal geometry would be done with the use of shell theory \cite{Shell}. These models usually are quite complex and hence difficult to analyse. There are also models that use Zernike polynomials \cite{Zernike}. These are certain orthogonal polynomials that are used to model abberations in cornea or lens.

In this paper we present another model of corneal geometry based on a nonlinear membrane equation. Assuming radial symmetry we reduce it to a one-dimensional case and then provide some estimates on the exact solution. When fitting with data we use its simplified form and find that mean error is of order of a few per cent. 

\section{Derivation}
Cornea is a biological structure which is similar to a shell. Its diameter is about $11mm$ and its thickness is approximately $0.5mm$ in the center and $0.7mm$ in the peripheral part \cite{Cornea}. Thanks to its structural parameters cornea provides a great protection for more sensitive, inner parts of the eye. We will distinguish only interior and exterior surface of cornea although it is made of five layers. In order from the most exterior they are as follows: epithelium, Bowman's layer, stroma, Descement's membrane and endothelium. Stroma is the thickest layer sometimes identified with cornea itself.

We would like to derive an equation describing the steady state shape of cornea's surface (either interior or exterior). We will describe this geometry by a function $h=h(x,y)$. Assume that cornea is a two-dimensional thin membrane, that is, it has no bending or twisting moments and its tension $T$ is constant at every point. Moreover, impose a pressure $P$ acting from below and normally to the surface. To model cornea's elastic properties we propose an additional restoring force which is proportional to the deflection $h$. Finally, let $\Omega$ be the domain in $xy$ plane on which this membrane is situated.

It can be shown (see \cite{Tik}) that a steady state membrane equation is of the form
\begin{equation}
	-T\Delta h=F,
\end{equation}
where $F$ is the net force density (pressure) acting on the membrane. In our case the equation will be
\begin{equation}
	-T\Delta h+k h=\frac{P}{\sqrt{1+\left\|\nabla h\right\|^2}}. 
	\label{dimensional}
\end{equation}
This is a second order nonlinear partial differential equation. The square root on the right-hand-side comes from the projection of normal vector onto the vertical $z$-axis. For the boundary conditions we impose
\begin{equation}
	h=0 \quad on \quad \partial\Omega.
\end{equation}
Appearing constants have following dimensions: $[T]=\frac{N}{m}=M T^{-2}$, $[k]=\frac{N}{m^3}=M L^{-2} T^{-2}$ and $[P]=\frac{N}{m^2}=M L^{-1} T^{-2}$, where $M$, $L$, $T$ are dimensions of mass, length and time respectively. Also, let $R$ be a typical linear size of the cornea (for example its radius). Taking
\begin{equation}
	h^*=\frac{h}{R}, \quad x^*=\frac{x}{R},
\end{equation}
we transform (\ref{dimensional}) into a nondimensional equation (dropping asterisks for clarity)
\begin{equation}
	-\Delta h+ah=\frac{b}{\sqrt{1+\left\| \nabla h \right\|^2}} \quad on \quad \Omega,
	\label{eqlap}
\end{equation}
where we denoted $a:=\frac{k R^2}{T}$ and $b:=\frac{P R}{T}$. From physical reasoning we allow only positive values of $a$ and $b$. We will see in Section \ref{fitting} that these parameters are both of order of unity.

\section{Model Analysis}

From now on $\Omega$ will be an unit circle in the $xy$-plane. We assume that cornea has radial symmetry, thus $h$ depends only on radius $r$ and in polar coordinates $h=h(r)$. Rewriting (\ref{eqlap}) we get
\begin{equation}
 -\frac{1}{r}\frac{d}{dr}\left( r \frac{dh}{dr} \right)+ah=\frac{b}{\sqrt{1+ h'^2}} \quad 0\leq r\leq1
 \label{eqpolar}
\end{equation}
with following boundary conditions
\begin{equation}
	h'(0)=0, \quad h(1)=0.
	\label{boundary}
\end{equation}
The condition $h'(0)=0$ guarantees absence of singularities and smoothness at the origin. 

We transform (\ref{eqpolar}) into an integral equation 
\begin{equation}
	h(r)=\frac{b}{I_0(\sqrt{a})}\left(v_0(r)\int_{r}^{1}{t v_1(t) P(h'(t))dt} + v_1(r)\int_{0}^{r}{t v_0(t) P(h'(t))dt} \right),
	\label{eqint}
\end{equation}
where $v_0$, $v_1$ and $P$ are defined by
\begin{equation}
\begin{split}
	&v_0(r):=I_0(\sqrt{a}r), \quad v_1(r):=I_0(\sqrt{a})K_0(\sqrt{a}r)-I_0(\sqrt{a}r)K_0(\sqrt{a}), \\
	&P(x)=\frac{1}{\sqrt{1+x^2}},
	\label{eqv}
\end{split}
\end{equation}
and $I_\nu$, $K_\nu$ ($\nu\in \mathbb{R}$) are $\nu$-order modified Bessel functions of first and second kind respectively. From elementary properties of Bessel functions (see \cite{AS}) we can deduce the following estimates
\begin{equation}
\begin{split}
	&v_0(r)>0, \quad v_1(r) \geq 0, \quad v'_0(r)=\sqrt{a} I_1(\sqrt{a}r) \geq 0, \\
  &v'_1(r)=-\sqrt{a}\left(I_0(\sqrt{a})K_1(\sqrt{a}r)+I_1(\sqrt{a}r)K_0(\sqrt{a})\right)\leq0.
  \label{bessel}
\end{split}
\end{equation} 
Thus, we have $h\geq0$. Also, it can be shown that
\begin{equation}
	\frac{d}{dr}\left(-r v'_1(r)\right)=-a r\left(I_0(\sqrt{a})K_0(\sqrt{a}r)-K_0(\sqrt{a})I_0(\sqrt{a}r)\right)\leq0,
	\label{rv'}
\end{equation}
hence the function $-r v'(r)$ is nonincreasing. Since $K_{\nu}(z)\sim\frac{1}{2}\Gamma(\nu)(\frac{1}{2}z)^{-\nu}$ we have $\lim_{r\rightarrow 0^+}{\left(-r v'(r)\right)}=I_0(\sqrt{a})$. Moreover, notice the following useful formulas \cite{AS}:
\begin{equation}
	I_\nu(z)K_{\nu+1}(z)+I_{\nu+1}(z)K_{\nu}(z)=\frac{1}{z}
	\label{Wronskian}
\end{equation}
and
\begin{equation}
\begin{split}
	z K_0(z)=-\frac{d}{dz} \left(z K_1(z)\right), \quad z I_0(z)=\frac{d}{dz} \left(z I_1(z)\right), \\
	\frac{d}{dz}I_0(z)=I_1(z), \quad \frac{d}{dz}K_0(z)=-K_1(z).
\end{split}
	\label{pochodne}
\end{equation}
Finally, it is easy to calculate that 
\begin{equation}
	h'(r)=\frac{b}{I_0(\sqrt{a})}\left(v'_0(r)\int_{r}^{1}{t v_1(t) P(h'(t))dt} + v'_1(r)\int_{0}^{r}{t v_0(t) P(h'(t))dt} \right).
	\label{eqint'}
\end{equation}
The question about existence and uniqueness of solution to (\ref{eqpolar}) is answered by the following theorem which uses Picard-Lindeloeff iterations.

\begin{thm}
Define the approximation sequence by
\begin{equation}
	h_0(r):=\frac{b}{a}\left(1-\frac{I_0(\sqrt{a}r)}{I_0(\sqrt{a})}\right),
	\label{h0}
\end{equation}

\begin{equation}
	h_n(r):=\frac{b}{I_0(\sqrt{a})}\left(v_0(r)\int_{r}^{1}{t v_1(t) P(h_{n-1}'(t))dt} + v_1(r)\int_{0}^{r}{t v_0(t) P(h_{n-1}'(t))dt} \right),
	\label{hn}
\end{equation}
where $n=1,2,...$ and $0\leq r \leq1$. The sequence $h_n$ is uniformly convergent to the unique solution of boundary value problem (\ref{eqpolar})-(\ref{boundary}) provided that
\begin{equation}
b<\frac{3\sqrt{3}}{2}\frac{\sqrt{a} I_0(\sqrt{a})}{I_1(\sqrt{a})\left( 2 I_0{\sqrt{a}}-1 \right)}.
\end{equation}
\label{exuniq}
\end{thm}

\begin{proof}
First, observe that P is Lipschitz continuous: $\left|P(x)-P(y)\right|\leq M \left|x-y\right|$, where $M=\frac{2}{3\sqrt{3}}$. We can write (\ref{eqint}) and (\ref{eqint'}) as $h(r)=\int_0^1{F(r,t)P(h'(t))dt}$ and $h'(r)=\int_0^1{G(r,t)P(h'(t))dt}$, where
\begin{equation}
\begin{split}
	F(r,t):=\frac{b}{I_0(\sqrt{a})}\left( t v_0(r)v_1(t) \chi_{[1,r]}(t) + t v_0(t)v_1(r) \chi_{[0,r]}(t) \right), \\
	G(r,t):=\frac{b}{I_0(\sqrt{a})}\left( t v'_0(r)v_1(t) \chi_{[1,r]}(t) + t v_0(t)v'_1(r) \chi_{[0,r]}(t) \right).
\end{split}
\end{equation}
It is an easy calculation using (\ref{rv'})-(\ref{pochodne}) to show that $\int_0^1{\left|F(r,t)\right|dt}\leq Q$ and  $\int_0^1{\left|G(r,t)\right|dt}\leq R$, where
\begin{equation}
\begin{split}
	Q:=\frac{b}{a}\left(1-\frac{1}{I_0(\sqrt{a})}\right), \\
	R:=\frac{b}{\sqrt{a}}\frac{I_1(\sqrt{a})}{I_0(\sqrt{a})}\left( 2 I_0{\sqrt{a}}-1 \right).
\end{split}
\end{equation}

Let $h_n(r)=h_0(r)+\sum_{i=1}^{n}{\left(h_i(r)-h_{i-1}(r)\right)}$ and similarly $h'_n(r)=h'_0(r)+\sum_{i=1}^{n}{\left(h'_i(r)-h'_{i-1}(r)\right)}$. We have to show that these series are uniformly convergent when $n\rightarrow\infty$. By estimate
\[
\begin{split}
	\left|h'_1(r)-h'_0(r)\right| \leq \int_0^1{\left|G(r,t)\right| \left|P(h'_0(t))-P(0)\right|dt} \\
	 \leq M\int_0^1{\left|G(r,t)\right| \left|h'_0(t)\right|dt}\leq MR^2,
\end{split}
\]
and generally
\[
\begin{split}
	&\left|h'_n(r)-h'_{n-1}(r)\right| \leq \int_0^1{\left|G(r,t)\right| \left|P(h'_{n-1}(t))-P(h'_{n-2}(t))\right|dt} \\ 
  &\leq M\int_0^1{\left|G(r,t)\right| \left|h'_{n-1}(t)-h'_{n-2}(t)\right| dt} \leq MR\left\|h'_{n-1}-h'_{n-2}\right\|_{\infty}, 
\end{split}
\]
we get $\left\|h'_n-h'_{n-1}\right\|_{\infty} \leq MR\left\|h'_{n-1}-h'_{n-2}\right\|_{\infty} \leq R(MR)^n$, where $\left\|.\right\|_{\infty}$ is a supremum norm. Moreover,
\[
	\left|h_n(r)-h_{n-1}(r)\right| \leq M\int_0^1{\left|F(r,t)\right| \left|h'_{n-1}(t)-h'_{n-2}(t)\right|dt}\leq Q (MR)^n,
\]
that is $\left\|h_n-h_{n-1}\right\|_{\infty} \leq Q (MR)^n$.
By the assumption $MR<1$, so the dominant number series are uniformly convergent. From uniform Cauchy-sequence argument, $h_n\rightarrow f$ and $h'_n\rightarrow f'$ for some $f$. The equality $f=h$ follows from continuity of $P$ and (\ref{hn}). This ends the proof.
\end{proof}

Imposing more restrictions on $b$ we can prove some estimates for the solution to (\ref{eqpolar}).

\begin{lem}
Assume that $b\leq\frac{\sqrt{a}}{I_1(\sqrt{a})}\frac{\sqrt{2I_0(\sqrt{a})-1}}{I_0(\sqrt{a})-1}$, then 
\begin{equation}
	\left(2-\frac{1}{I_0(\sqrt{a})}\right)h'_0\leq h' \leq 0.
\end{equation}
\label{lemest}
\end{lem}

\begin{proof}
Take $b$ as in the assumptions. By (\ref{bessel})-(\ref{eqint'}) and since $P\leq1$ we have the estimate
\begin{equation}
\begin{split}
	h'(r)&\leq\frac{b}{I_0(\sqrt{a})} v'_0(r)\int_{r}^{1}{t v_1(t) P(h'(t))dt} \\ 
			 &\leq\frac{b I_1(\sqrt{a}r)}{I_0(\sqrt{a})}\left(r\left(I_0(\sqrt{a})K_1(\sqrt{a}r)+I_1(\sqrt{a}r)K_0(\sqrt{a})\right) -\frac{1}{\sqrt{a}} \right).
\end{split}
\end{equation}
Now, by (\ref{rv'}) and our assumption on $b$ we deduce that
\begin{equation}
	h'(r)\leq\frac{b I_1(\sqrt{a})}{\sqrt{a} I_0(\sqrt{a})}\left(I_0(\sqrt{a})-1\right)\leq\frac{\sqrt{2I_0(\sqrt{a})-1}}{I_0(\sqrt{a})-1}.
\end{equation}
Next, by similar reasoning we have
\begin{equation}
\begin{split}
	-h'(r)&\leq-\frac{b}{I_0(\sqrt{a})}v'_1(r)\int_0^r{t v_0(r)P(h'(t))dt} \leq-\frac{b I_1(\sqrt{a}r)}{\sqrt{a} I_0(\sqrt{a})}r v'_1(r) \\
				&\leq\frac{b I_1(\sqrt{a})}{\sqrt{a}}\leq\frac{\sqrt{2I_0(\sqrt{a})-1}}{I_0(\sqrt{a})-1}.
\end{split}
\end{equation}
We have shown that $\left|h'(r)\right|\leq \frac{\sqrt{2I_0(\sqrt{a})-1}}{I_0(\sqrt{a})-1}$, thus by (\ref{eqv})
\begin{equation}
	C\leq P(h'(t)) \leq 1,
	\label{P}
\end{equation}
where $C=1-\frac{1}{I_0(\sqrt{a})} (0<C<1)$. Finally, from (\ref{rv'}) and (\ref{P}) we deduce that
\begin{equation}
\begin{split}
	h'(r)&\leq\frac{b}{I_0(\sqrt{a})}\left(v'_0(r)\int_{r}^{1}{t v_1(t)dt} + Cv'_1(r)\int_{0}^{r}{t v_0(t)dt} \right) \\
			 &=\frac{b I_1(\sqrt{a}r)}{I_0(\sqrt{a})}\left((1-C)r\left(I_0(\sqrt{a})K_1(\sqrt{a}r)+K_0(\sqrt{a})I_1(\sqrt{a}r)\right) -\frac{1}{\sqrt{a}} \right)\leq0.
\end{split}
\end{equation}
Similarly, estimating from below we can prove
\begin{equation}
\begin{split}
	-h'(x)&\leq-\frac{b}{I_0(\sqrt{a})}\left(Cv'_0(r)\int_{r}^{1}{t v_1(t)dt} + v'_1(r)\int_{0}^{r}{t v_0(t)dt} \right) \\
				&=\frac{b I_1(\sqrt{a}r)}{I_0(\sqrt{a})}\left( (1-C)r\left(I_0(\sqrt{a})K_1(\sqrt{ar})+K_0(\sqrt{a})I_1(\sqrt{a}r)\right)+\frac{C}{\sqrt{a}} \right) \\
				&\leq\frac{b I_1(\sqrt{a}r)}{\sqrt{a}I_0(\sqrt{a})}\left(2-\frac{1}{I_0(\sqrt{a})} \right)=-\left(2-\frac{1}{I_0(\sqrt{a})}\right)h'_0(r).
\end{split}
\end{equation}
Therefore we have proved our claims.
\end{proof}

\begin{cor}
	Let $b\leq\frac{\sqrt{a}}{I_1(\sqrt{a})}\frac{\sqrt{2I_0(\sqrt{a})-1}}{I_0(\sqrt{a})-1}$. Solution of the problem (\ref{eqpolar})-(\ref{boundary}) is a positive, nonincreasing function $h$ with
\begin{equation}
	A h_1\leq h \leq h_0,
\end{equation}
where $h_0$ and $h_1$ are defined by (\ref{h0})-(\ref{hn}), while 
\begin{equation}
A=\frac{1+h'_0(1)^2}{1+\left(2-\frac{1}{I_0(\sqrt{a})}\right)h'_0(1)^2}.
\end{equation}
\label{estimates}
\end{cor}

\begin{proof}
Lemma \ref{lemest} gives us the fact that $h$ is nonincreasing. Moreover, we have $P\left(\left(2-\frac{1}{I_0(\sqrt{a})}\right)h'_0\right)\leq P(h)\leq 1$. Simple calculus allows us to deduce that $A P\left(h'_0\right)\leq P\left(\left(2-\frac{1}{I_0(\sqrt{a})}\right)h'_0\right)$. Now using the equation for $h$ (\ref{eqint}) we can conclude the proof.
\end{proof}
Various estimates on $h$ are depicted on Figure \ref{estpic}. The convergence of $h_n$ is very fast. In fact, only after 4 iterations we have $\left\| h_4-h_3\right\|_\infty\approx10^{-6}$. Although Picard's iterations can be very demanding on computer power they converge very rapidly hence the number of iterations can be kept small.

\begin{figure}[htb!]
	\centering
	\includegraphics[trim = 40mm 90mm 40mm 90mm, scale=0.5]{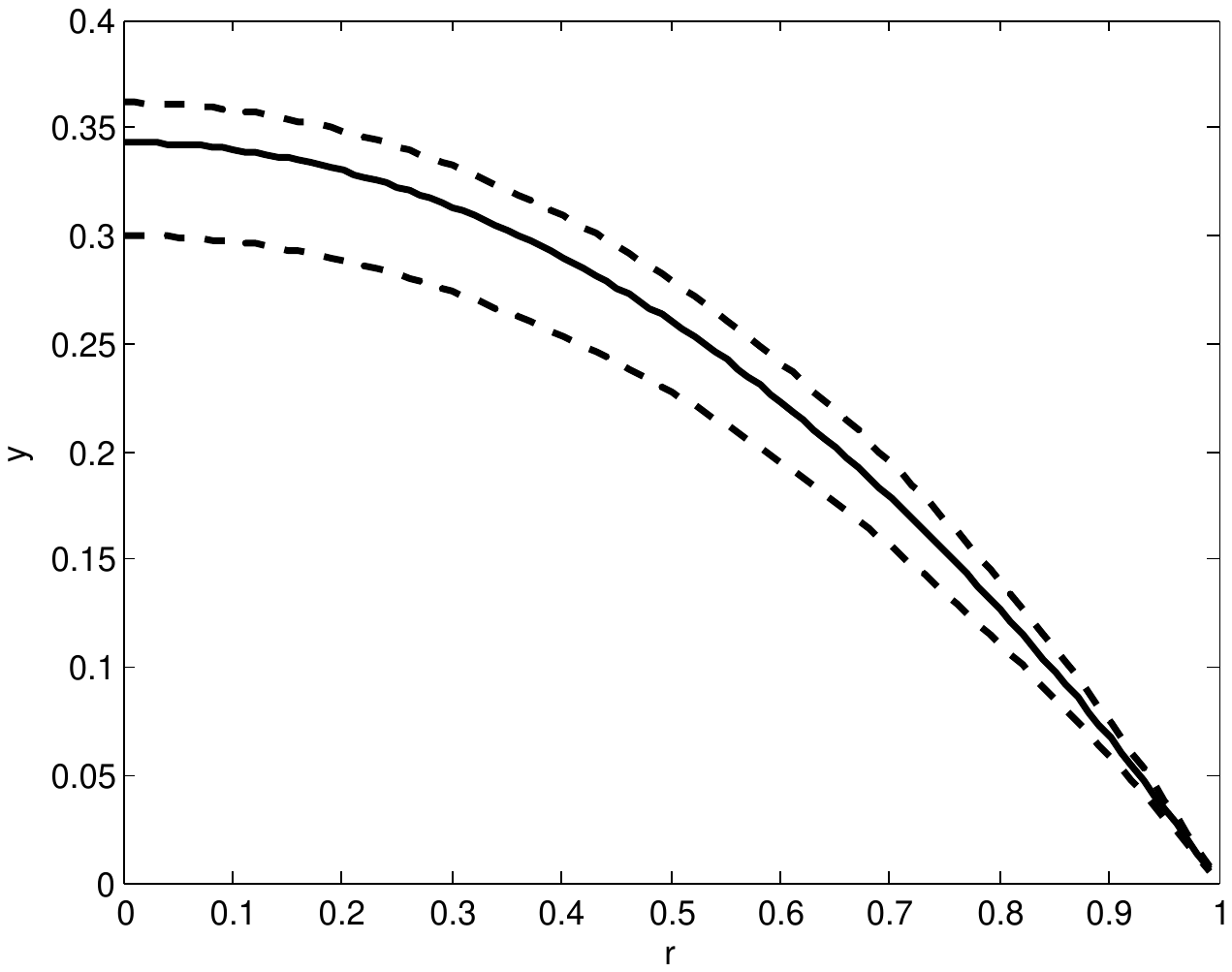}
	\caption{Estimates on $h$ for $a=b=2$. Dashed line: $h_0$ (top) and $A h_1$ (bottom). Solid line: $h_4$.} 
	\label{estpic}
\end{figure} 

To provide validity of Theorem $\ref{exuniq}$ and Corollary \ref{estimates} we have to restrict the values of $b$. Nevertheless, we will see in the next section that values of $(a,b)$ calculated from the data lie in the admissible region. Figure \ref{ab} shows the assumed restrictions on $b$.

\begin{figure}[htb!]
	\centering
	\includegraphics[scale=0.7]{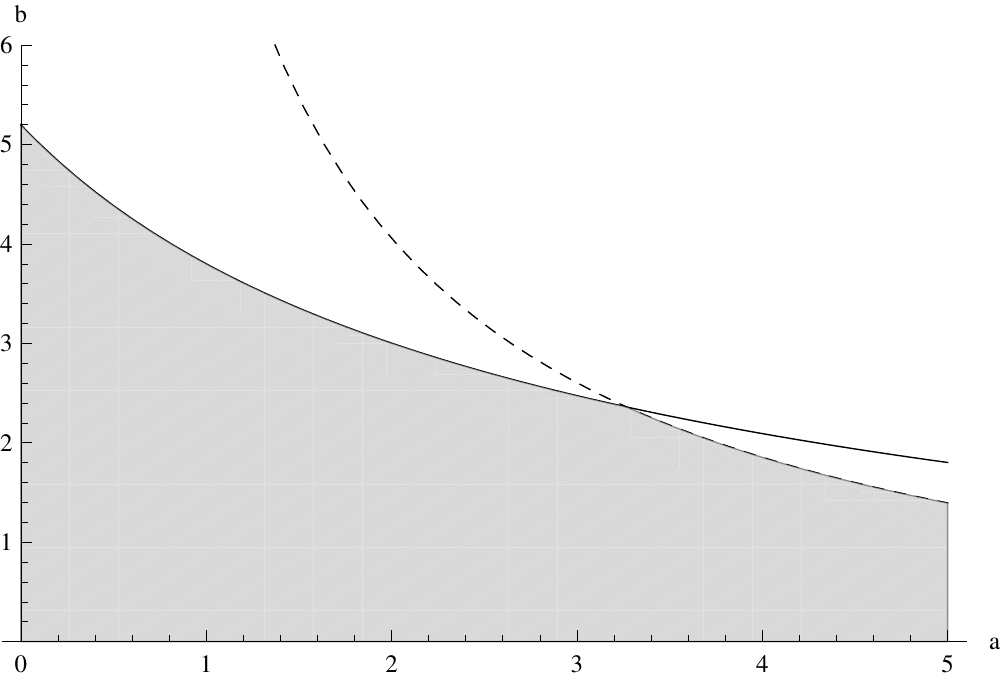}
	\caption{Values of $(a,b)$ assumed in Theorem \ref{exuniq} and Corollary \ref{estimates}. Solid line: $\frac{3\sqrt{3}}{2}\frac{\sqrt{a} I_0{\sqrt{a}}}{I_1{\sqrt{a}}\left( 2 I_0{\sqrt{a}}-1 \right)}$. Dashed line: $\frac{\sqrt{a}}{I_1(\sqrt{a})}\frac{\sqrt{2I_0(\sqrt{a})-1}}{I_0(\sqrt{a})-1}$.}
	\label{ab}
\end{figure}

\section{Data Fitting}
\label{fitting}
In this section we present the results of data fitting. Instead of exact solution of (\ref{eqpolar}) we will use its zeroth-order approximation $h_0$ which is defined by (\ref{h0}). The use of $h_0$ rather than $h$ is equivalent to making the assumption that the pressure acts in the $z$-direction as apposed to the normal direction. Moreover, $h_0$ has a simple analytic form and hence is much easier to use in applications. 

We will determine constants $a$ and $b$ by taking two measurements: cornea's maximum deflection and central radius of curvature. These two values are readily measurable quantities and allow us to make a fit on the whole domain. Suppose we know maximum deflection $h_0(0)$ and central radius of curvature $\rho(0)$, where $\rho(x)=\frac{(1+(h_0'(x))^2)^{\frac{3}{2}}}{|h_0''(x)|}$. Because $h'_0(0)=0$ we have $\rho(0)=\frac{2 I_0(\sqrt{a})}{b}$, thus
\begin{equation}
	b=\frac{2 I_0(\sqrt{a})}{\rho(0)},
\end{equation} 
where in calculating the second derivative we used $2I'_1(z)=I_0(z)+I_2(z)$ (see \cite{AS}). Finally, using (\ref{h0}) we deduce that $a$ is the solution of the following equation
\begin{equation}
	\frac{1}{2}h_0(0)\rho(0)a-I_0(\sqrt{a})+1=0.
\end{equation}
For two datasets consisting of $123\times123$ points representing interior and exterior surfaces of the cornea we calculated that $a_{in}\approx 2.07883$, $b_{in}\approx 2.76741$, $a_{ex}\approx 1.94398$ and $b_{ex}\approx2.27534$. It is evident that these points are within the admissible region depicted on Fig. \ref{ab}.

To make the fitting more accurate we can resign from assumed radial symmetry by making a slight perturbation in the domain $\Omega$. Specificly, as $\Omega$ we will use ellipse instead of a circle. The eccentricity of this ellipse can be calculated from the data, for example, by measuring the eccentricity of level curves of the corneal shape. In order to reshape our $\Omega$ into an ellipse we have to change $r=\sqrt{x^2+y^2}$ into $r=\sqrt{\frac{x^2}{R_1^2}+\frac{y^2}{R_2^2}}$, where $R_1$ and $R_2$ are semiaxes of the ellipse in the domain.

Mean fitting errors for interior and exterior surfaces are respectively $0.035mm$ $(3.6\%)$ and $0.014mm$ $(1.4\%)$ and squared eccentricities are $-0.0214$ and $0.0234$. The interior surface is thus slightly prolate while exterior is oblate. It is evident that deviations from circular shape are very minute. The magnitude of absolute fitting error is shown on Fig. \ref{3D}.

\begin{figure}[htb!]
	\centering
	\includegraphics[scale=0.5]{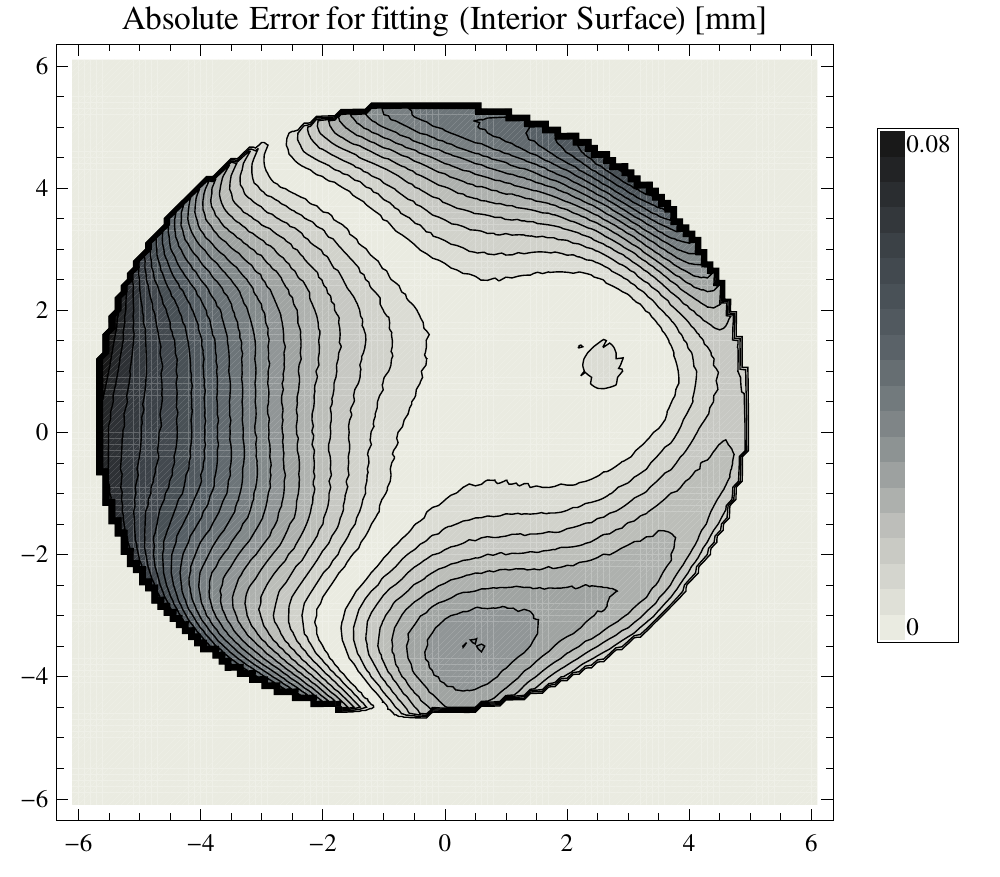}
	\includegraphics[scale=0.5]{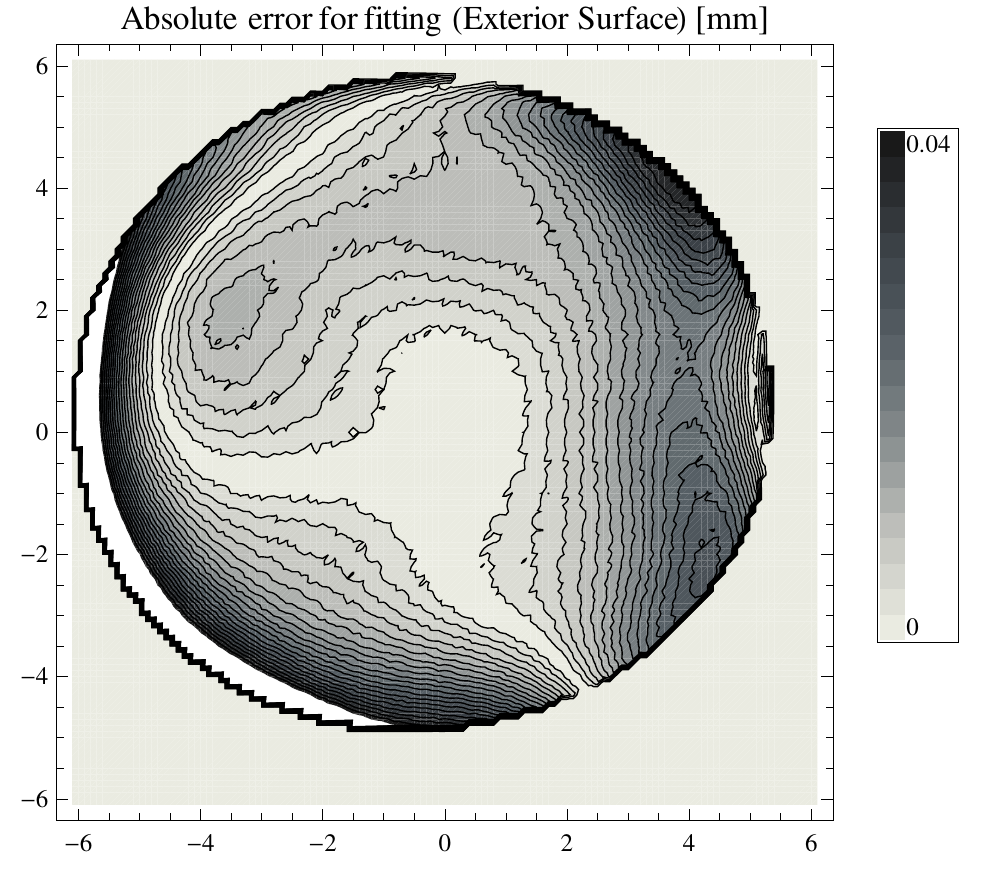}
	\caption{Absolute fitting errors for interior (left) and exterior (right) surfaces.}
	\label{3D}
\end{figure} 

Another important feature of cornea's geometry is its curvature. For measuring radius of curvature optometrists use the following quantity known as axial distance (\cite{Klein})
\begin{equation}
	d=r \sqrt{1+\frac{1}{h_r^2+\frac{1}{r^2}h_{\theta}^2}}=\sqrt{x^2+y^2} \sqrt{1+\frac{1}{h_x^2+h_y^2}}.
\end{equation}
In the case of surface of revolution, normal vector lies in the meridional plane so $d$ is the distance along the normal vector from the surface point to the axis of revolution. The absolute error in fitting $d$ calculated for our $h_0$ is depicted on Fig. \ref{cur}. Mean errors are approximately $0.146mm$ $(2.1\%)$ and $0.14mm$ $(1.7\%)$ for interior and exterior surfaces respecively.

\begin{figure}[htb!]
	\centering
	\includegraphics[scale=0.5]{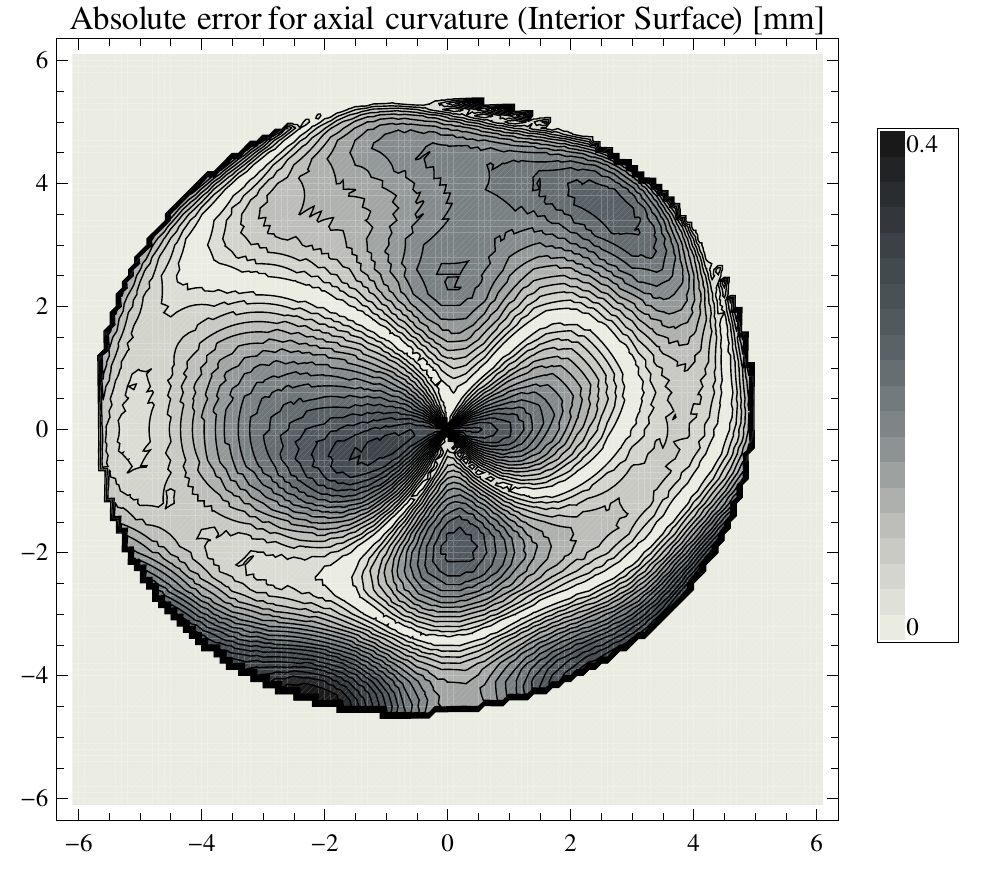}
	\includegraphics[scale=0.5]{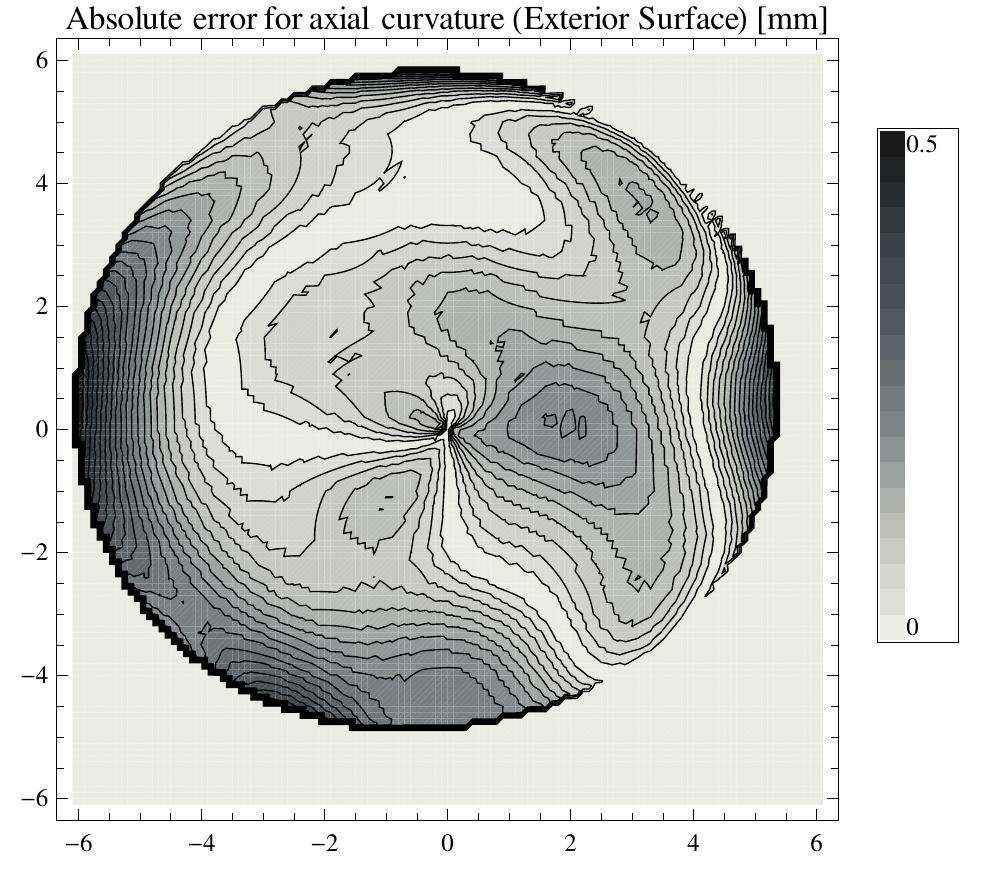}
	\caption{Absolute errors in axial distance for interior (left) and exterior (right) surfaces.}
	\label{cur}
\end{figure} 

\section{Conclusion and discussion}
We have proposed a nonlinear mathematical model of corneal geometry based on a thin membrane equation. The existence and uniqueness of the solution was proved provided that one of the parameters was bounded in respect the the other. Various estimates (Lemma \ref{lemest} and Corollary \ref{estimates}) showed that the solution behaves according to common physical intuition. Approximating sequence (\ref{hn}) converges very rapidly to the exact solution what can be deduced from its definition and numerical experiments.

We have fitted zeroth-order approximation $h_0$ to the two meshes of $123\times123$ points representing interior and exterior surfaces of cornea. We have slightly perturbed its radial symmetry according to the measurable quantities. The mean fitting error was of order of a few per cent which provided a good fit. The error was slightly smaller for exterior than interior surface. Also, the function known as axial distance representing cornea's radius of curvature was in a good accordance with the data. Again, the magnitude of errors were of order of a few per cent what allows us to conclude that this is a good result for a model of that simplicity.

\section*{Acknowledgment}
Authors would like to thank Dr. Robert Iskander from Institute of Biomedical Engineering
and Instrumentation, Wroclaw University of Technology, Poland and School of Optometry, Queensland University of Technology, Australia for access to the data.


\begin{thebibliography}{00}
\bibitem{Conic} H. Kasprzak, D.R. Iskander, Approximating ocular surfaces by generalized conic curves, Ophthal. Physiol. Opt., 26:602-609, 2006.
\bibitem{Shell} K. Anderson, A. El-Sheikh, T. Newson, Application of structural analysis to the
mechanical behaviour of the cornea, J. R. Soc. Interface 1, 3-15, 2004.
\bibitem{Zernike} D.R. Iskander, M.J. Collins, B. Davis, Optimal Modeling of Corneal Surfaces by Zernike Polynomials, IEEE Transactions on Biomedical Engineering, Vol. 48, No. 1, 2001.
\bibitem{Cornea} W. Trattler, P. Majmudar, J. I. Luchs, T.Swartz, Cornea Handbook, Slack Incorporated, 2010.
\bibitem{Tik} A. N. Tikhonov, A. A. Samarskii, Equations of Mathematical Physics, Dover Publications, 1963.
\bibitem{AS} M. Abramowitz, I. A. Stegun, Handbook of Mathematical Functions: with Formulas, Graphs, and Mathematical Tables, Dover Publications, 1965.
\bibitem{Klein} S. A. Klein, Axial curvature and the Skew Ray Error in Corneal Topography, Optometry and Vision Science, Vol. 74, No. 11, PP. 931-944, 1997.
\end{thebibliography}
\end{document}